\documentclass{article}

\usepackage{amsmath,amssymb,amsthm,multirow}

\newcommand{\N}{\mathbb{Z}^+}
\newcommand{\Q}{\mathbb{Q}}
\newcommand{\Z}{\mathbb{Z}}
\newcommand{\cS}{\mathcal{S}}
\newcommand{\ri}{\mathcal{O}}
\newcommand{\fr}{\mathfrak}

\newtheorem{thm}{Theorem}
\newtheorem{lemma}[thm]{Lemma}

\begin{document}

\title{On the Diophantine equation $NX^2 + 2^{L}3^{M} = Y^N$}

\author{Eva G. Goedhart\thanks{Department of Mathematics, Bryn Mawr College, Bryn Mawr, PA 19010} \and Helen G. Grundman\thanks{Department of Mathematics, Bryn Mawr College, Bryn Mawr, PA 19010}}
\date{\today}

\maketitle

\begin{abstract}
We prove that the Diophantine equation 
$NX^2 + 2^L 3^M = Y^N$ has no solutions $(N,X,Y,L,M)$
in positive integers with $N > 1$ and $\gcd(NX,Y) = 1$, 
generalizing results of Luca, Wang and Wang, 
and Luca and Soydan. Our proofs use results
of Bilu, Hanrot, and Voutier on defective Lehmer pairs.
\end{abstract}\vspace{1pc}

\section{Introduction}

In this work, we prove the following theorem.

\begin{thm}
\label{mainthm}
The equation
\begin{equation}
\label{maineq}
NX^2 + 2^{L}3^{M} = Y^N,
\end{equation}
has no solution with $N$, $X$, $Y$, $L$, $M\in\N$, $N > 1$,
and $\gcd(NX,Y) = 1$.
\end{thm}

Equation~(\ref{maineq}) is a variation of the equation
$NX^2 + 2^K = Y^N$ studied by
Wang and Wang~\cite{WW11} and by
Luca and Soydan~\cite{LS12} and of the equation
$X^2 + 2^L3^M = Y^N$ studied by
Luca~\cite{Lu02}.
Our proofs draw upon ideas from each of these papers.

We begin by showing that it suffices to prove Theorem~\ref{mainthm} in 
the case where $N$ is square-free, and by reviewing a needed result on
Lehmer pairs.
In Section~\ref{evensec}, we prove the special case of Theorem~\ref{mainthm}
in which both
of the exponents $L$ and $M$ are assumed to be even.  Then in
Section~\ref{oddsec}, we prove the remaining cases,
thus completing the proof of Theorem~\ref{mainthm}.

\begin{lemma}
\label{sqfree}
If there exists a solution to $NX^2 + 2^{L}3^{M} = Y^N$
as in Theorem~\ref{mainthm}, then there exists a solution with
the same values of $L$ and $M$, but 
with $N$ square-free.
\end{lemma}

\begin{proof}
Suppose that $(N,X,Y,L,M) = (n,x,y,\ell,m)$ is a solution to
$NX^2 + 2^{L}3^{M} = Y^N$, with
$n$, $x$, $y$, $\ell$, $m\in\N$, $n > 1$, 
and $\gcd(nx,y) = 1$.
Note that $\ell$, $m > 0$ implies that $\gcd(n,6) = 1$.

Let $n = uv^2$, with $u$, $v \in \N$ and $u$ square-free.
Suppose that $u = 1$.
Then $(N,X,Y,L,M) = (n,vx,y,\ell,m)$ is a solution to $X^2 + 2^L3^M = Y^N$
with $N$, $X$, $Y$, $L$, $M \in \N$, $N > 1$, and $\gcd(X,Y) = 1$.
By~\cite[Theorem 2.1]{Lu02}, this implies that
$n = N = 3$ or $4$, contradicting that $\gcd(n,6) = 1$.
Thus $u > 1$.

Now, note that $u(vx)^2 + 2^\ell3^{m} = y^n = (y^{v^2})^u$,
and so $(N,X,Y,L,M) = (u,vx,y^{v^2},\ell,m)$
is a solution to~(\ref{maineq})
with $\gcd(NX,Y) = \gcd(uvx,y^{v^2}) = 1$, and $N = u > 1$.
\end{proof}

A key element in our proofs is the theory of Lehmer sequences and
defective Lehmer pairs, which we now briefly describe. 
For a more detailed introduction, see~\cite{Vo95}.

A pair of algebraic integers $(\gamma,\delta)$ is called a
Lehmer pair if 
$\gamma\delta\in\Z - \{0\}$, $(\gamma+\delta)^2\in\Z - \{0\}$,
$\gcd(\gamma\delta,(\gamma+\delta)^2) = 1$, and 
$\frac{\gamma}{\delta}$ is not a root of unity.
Given a Lehmer pair, $(\gamma,\delta)$, and
$s\in\N$, define
\[L_s(\gamma,\delta)=\left\{\begin{array}{ll}
\frac{\gamma^s-\delta^s}{\gamma-\delta_{\ }}, & \mbox{ if $s$ is odd,}\\
\frac{\gamma^s-\delta^s}{\gamma^2-\delta^2},  & \mbox{ if $s$ is even.}
\end{array}
\right.\]
The Lehmer pair $(\gamma,\delta)$ is $s$-defective if, for each
$p\mid L_s(\gamma,\delta)$, 
\[p\mid 
(\gamma^2 - \delta^2)^2L_1(\gamma,\delta)\dots L_{s-1}(\gamma,\delta).\]

We need the following lemma~\cite[Theorem 1(ii)]{Vo95}.

\begin{lemma}[Voutier]\label{VoThm1}
Let $s\in \N$ such that $6 < s\leq 30$ and 
$s\neq 8$, $10$, or $12$.  
If $(\gamma,\delta)$ is an $s$-defective Lehmer pair, then 
for some $k\in \{0,1,2,3\}$, $i^k\gamma$ is one of the 
values listed in Table~\ref{Votable}.
\end{lemma}

\begin{table}[hbt] \caption{Possible values of $i^{k}\gamma$ 
in Lemma~\ref{VoThm1}.
\label{Votable}}
\centering
     \begin{tabular}{ | c | c | c | c |}
    \hline
    $s$ & \multicolumn{3}{|c|}{$i^k\gamma$, for $k\in \{0,1,2,3\}$}  \\ 
    \hline
    \multirow{2}{*}{$7$} & $\frac{1 \pm \sqrt{-7}}{2}$ & $ \frac{1 \pm \sqrt{-19}}{2}$ & $ \frac{\sqrt{3} \pm \sqrt{-5}}{2}$\\ \cline{2-4}& $ \frac{\sqrt{5} \pm \sqrt{-7}}{2}$ & $ \frac{\sqrt{13} \pm \sqrt{-3}}{2}$ & $ \frac{\sqrt{14} \pm \sqrt{-22}}{2}$  \\ \hline
    $9$ & $\frac{\sqrt{5} \pm \sqrt{-3}}{2}$ & $ \frac{\sqrt{7} \pm \sqrt{-1}}{2}$ & $ \frac{\sqrt{7} \pm \sqrt{-5}}{2}$ \\ \hline
    $13$ & $\frac{1 \pm \sqrt{-7}}{2}$ & &\\ \hline
   \multirow{2}{*}{$14$} & $\frac{\sqrt{3} \pm \sqrt{-13}}{2}$ & $ \frac{\sqrt{5} \pm \sqrt{-3}}{2}$ & $ \frac{\sqrt{7} \pm \sqrt{-1}}{2}$\\\cline{2-4} & $ \frac{\sqrt{7} \pm \sqrt{-5}}{2}$ & $ \frac{\sqrt{19} \pm \sqrt{-1}}{2}$ & $ \frac{\sqrt{22} \pm \sqrt{-14}}{2}$\\ \hline
    $15$ & $\frac{\sqrt{7} \pm \sqrt{-1}}{2}$ & $ \frac{\sqrt{10} \pm \sqrt{-2}}{2}$ & \\ \hline
    $18$ & $\frac{1 \pm \sqrt{-7}}{2}$ & $ \frac{\sqrt{3} \pm \sqrt{-5}}{2}$ & $ \frac{\sqrt{5} \pm \sqrt{-7}}{2}$\\ \hline
    $24$ & $\frac{\sqrt{3} \pm \sqrt{-5}}{2}$ & $ \frac{\sqrt{5} \pm \sqrt{-3}}{2}$ &\\ \hline
    $26$ & $\frac{\sqrt{7} \pm \sqrt{-1}}{2}$ & &\\ \hline
    $30$ & $\frac{1 \pm \sqrt{-7}}{2}$ & $ \frac{\sqrt{2} \pm \sqrt{-10}}{2}$ &\\
    \hline
    \end{tabular}
\end{table}

\section{Even Exponents}
\label{evensec}

In this section, we prove the following special
case of Theorem~\ref{mainthm}.

\begin{thm}
\label{eventhm}
The equation
\[NX^2 + 2^{2L}3^{2M} = Y^N,\]
has no solution with $N$, $X$, $Y$, $L$, $M\in\N$, $N > 1$,
and $\gcd(NX,Y) = 1$.
\end{thm}

\begin{proof}
Suppose that $(N,X,Y,L,M) = (n,x,y,\ell,m)$ is a solution to
$NX^2 + 2^{2L}3^{2M} = Y^N$, with
$n$, $x$, $y$, $\ell$, $m\in\N$, $n > 1$, 
and $\gcd(nx,y) = 1$.  
It follows immediately that
$y > 1$ and $nx^2 \equiv y^n \pmod 6$.  Since $\gcd(nx,y) = 1$, we have
\[n\equiv y \equiv \pm 1 \pmod 6\hspace{.15in}\mbox{and}\hspace{.15in}
x\equiv \pm 1 \pmod 6.\]
By Lemma~\ref{sqfree}, we may assume that $n$ is square-free.

We now apply the following lemma, proved by Heuberger and Le~\cite{HL99}
and adapted to this form by Wang and Wang~\cite{WW11}.
\begin{lemma}[Heuberger \& Le]
Let $d\in\Z$ be square-free such that $d > 1$, and 
let $k\in \Z$ be odd such that
$k > 1$ and $\gcd(d,k) = 1$. Let 
$h(-4d)$ denote the number of classes of primitive
binary quadratic forms of discriminant $-4d$.
If the equation
\[ X^2 + dY^2 = k^Z,\hspace{.2in} X,Y,Z\in\Z,\hspace{.2in}\gcd(X,Y) = 1,
\hspace{.2in} Z > 0\]
has a solution, $(X,Y,Z)$, then 
there exist $X_1$, $Y_1$, $Z_1$, $t\in\N$ and
$\lambda_1$, $\lambda_2\in\{+1,-1\}$, such that
\[X_1^2+dY_1^2=k^{Z_1}, \hspace{.25in} \gcd(X_1,Y_1) = 1,\]
\[Z = Z_1t,\hspace{.25in}
Z_1\mid h(-4d),\hspace{.25in}\text{and}\]
\[X + Y\sqrt{-d} = \lambda_1(X_1+\lambda_2Y_1\sqrt{-d})^t.\]
\end{lemma}

By the lemma, since
$\left(2^{\ell}3^{m}\right)^2+nx^2=y^n$, with $n > 1$ square-free,
$y > 1$ odd, and $\gcd(nx,y)=1$,
there exist 
$X_1$, $Y_1$, $Z_1$, $t\in\N$ and
$\lambda_1$, $\lambda_2\in\{+1,-1\}$, such that
\begin{equation}
\label{X1Y1Z1}
X_1^2+nY_1^2=y^{Z_1}, \hspace{.25in} \gcd(X_1,Y_1) = 1,
\end{equation}
\begin{equation}
\label{Z1}
n = Z_1t,\hspace{.25in}
Z_1\mid h(-4n),\hspace{.25in}\text{and}
\end{equation}
\begin{equation}
\label{2l2mlem}
2^\ell 3^{m}+x\sqrt{-n}=\lambda_1(X_1+\lambda_2Y_1\sqrt{-n})^t.
\end{equation}
Note that, since $\gcd(n,6) = 1$, $\gcd(t,6) = \gcd(Z_1,6) = 1$.
Thus $t$ is odd and $y^{Z_1} \equiv y \equiv n \equiv \pm 1 \pmod 6$.
For ease in notation, let $t = 2t_1 + 1$.

Expanding~(\ref{2l2mlem})
and taking the absolute value of the real and imaginary parts 
of each side yields
\begin{equation}
\label{realeven}
2^\ell 3^{m} =
\left|\sum_{j=0}^{t_1} \binom{t}{2j} 
X_1^{t-2j}(-nY_1^2)^j\right| =
X_1\left|\sum_{j=0}^{t_1}\binom{t}{2j}X_1^{t-2j-1}(-nY_1^2)^j
\right|,
\end{equation}
and
\begin{equation}
\label{imageven}
x=Y_1\left|\sum_{j=0}^{t_1}\binom{t}{2j+1}
X_1^{ t-2j-1}(-nY_1^2)^{j}\right|. 
\end{equation}
By equation~(\ref{realeven}), $2$ and $3$ are the only possible prime
divisors of $X_1$. 
By equation~(\ref{imageven}) and $\gcd(x,6) = 1$, $\gcd(6, Y_1)=1$.
Thus $Y_1\equiv \pm 1\pmod 6$.

By equation~(\ref{X1Y1Z1}), 
$X_1^2 + n \equiv n \pmod 6$,
and so $X_1 \equiv 0 \pmod 6$.

Rewriting equation~(\ref{realeven}) as $2^\ell3^{m} = X_1|\cS|$, with
\[\cS = \sum_{j=0}^{t_1}\binom{t}{2j}X_1^{t-2j-1}(-nY_1^2)^j,\]
we have
\[\cS \equiv \binom{t}{t-1} (-nY_1^2)^{t_1} \equiv \pm1 \pmod 6.\]
But then, $\gcd(\cS,6) = 1$ and so 
$X_1 = 2^\ell3^{m}$ and $|\cS| = 1$.

Let $\gamma = X_1 + Y_1\sqrt{-n} $
and let $\delta = -X_1 + Y_1\sqrt{-n}$.

\begin{lemma} The pair
$(\gamma,\delta)$ is a $t$-defective Lehmer pair.
\end{lemma}

\begin{proof}
An easy calculation shows that
$\gamma\delta = -X_1^2 - nY_1^2 = -y^{Z_1}$ and
$(\gamma + \delta)^2 = -4nY_1^2$, each of which is nonzero.
Suppose that $p$ is prime such 
that $p\mid \gcd(\gamma\delta,(\gamma+\delta)^2)$.
Then, since $\gcd(n,y)=1$ and $y$ is odd, $p\mid Y_1$.  Additionally,
$p\mid (y^{Z_1}-nY_1^2)$ and so $p\mid X_1$.  But $\gcd(X_1,Y_1)=1$, and thus
$\gcd(\gamma\delta,(\gamma+\delta)^2) = 1$.  Note that
since $n > 1$, $\gcd(n,6) = 1$, and $n$ is square-free,
the only roots of unity in $\Q(\sqrt{-n})$ are $\pm 1$.  Thus,
$\frac{\gamma}{\delta}$ is not a root of unity. 
Therefore, $(\gamma,\delta)$ is a Lehmer pair.

Finally, by equations~(\ref{2l2mlem}) and~(\ref{realeven}),
\[|L_t(\gamma,\delta)| = 
\left|\frac{\gamma^t-\delta^t}{\gamma-\delta}\right| =
\left|\frac{2\Re(\gamma^t)}{2\Re(\gamma)}\right| = 
\frac{X_1|\cS|}{X_1} = 1.\]
Thus, $(\gamma,\delta)$ is a $t$-defective Lehmer pair.
\end{proof}

By the work of Bilu, Hanrot, and Voutier~\cite[Theorem 1.4]{BHV01}, since
there exists a $t$-defective Lehmer pair, we have that $t \leq 30$.
Then, using Lemma~\ref{VoThm1} with the fact that $\gcd(t,6) = 1$,
it follows that $t\in \{1,5\}$.

If $t = 5$, then
\[\cS = \sum_{j=0}^2\binom{5}{2j}(2^\ell 3^{m})^{5-2j-1}(-nY_1^2)^j =
2^{4\ell}3^{4m}-10\cdot2^{2\ell}3^{2m}nY_1^2+5 n^2Y_1^4.\]
Since $\cS = \pm 1$ and $n$ and $Y_1$ are both odd,
$\pm 1 = \cS \equiv 5 n^2Y_1^4 \equiv 5 \pmod 8$,
which is impossible.

Thus, $t = 1$.  So, by equation~(\ref{Z1}), $Z_1 = n$ and,
hence, $n\mid h(-4n)$.
But, since $n$ is greater than 1 and square-free, by~\cite[Lemma 3]{WW11},
$n > h(-4n)$, a contradiction.
\end{proof}

\section{Odd Exponents}
\label{oddsec}

In this section, we prove the remaining cases of Theorem~\ref{mainthm},
as described in the following theorem.

\begin{thm}
\label{oddthm}
The equation
\[NX^2 + 2^{L}3^{M} = Y^N,\]
has no solution with $N$, $X$, $Y$, $L$, $M\in\N$, $N > 1$,
$\gcd(NX,Y) = 1$, and $L$ and $M$ not both even.
\end{thm}

We begin with a basic computational lemma.

\begin{lemma}
\label{sums}
Let $t_1\in \Z$ and let $t = 2t_1 + 1$.  Then
\[\sum_{j=0}^{t_1}\binom{t}{2j+1} = 2^{t-1} 
\hspace{.3in}\mbox{and}\hspace{.3in}
\sum_{j=0}^{t_1}\binom{t}{2j+1}(-1)^j = \pm 2^{t_1}.\] 
\end{lemma}

\begin{proof}
First, let $f(t) = \sum_{j=0}^{t_1}\binom{t}{2j+1}$
and let $g(t) = \sum_{j=0}^{t_1}\binom{t}{2j}$.
Then $f(t) + g(t) = (1 + 1)^t$ and $-f(t) + g(t) = (1 - 1)^t$.
Solving these for $f(t)$ yields the first result.

Next, let $f_1(t) = \sum_{j=0}^{t_1}\binom{t}{2j+1}(-1)^j
= -i\sum_{j=0}^{t_1}\binom{t}{2j+1}(i)^{2j+1}$ and
$g_1(t) = -i\sum_{j=0}^{t_1}\binom{t}{2j}(i)^{2j}$.
Then $f_1(t) + g_1(t) = -i(1 + i)^t$ and $-f_1(t) + g_1(t) = -i(1 - i)^t$.
Solving for $f_1(t)$ completes the proof.
\end{proof}

\begin{proof}[Proof of Theorem~\ref{oddthm}]
Suppose that $(N,X,Y,L,M) = (n,x,y,\ell,m)$ is a solution to
$NX^2 + 2^{L}3^{M} = Y^N$, with
$n$, $x$, $y$, $\ell$, $m\in\N$, $n > 1$, $\gcd(nx,y) = 1$, and either $\ell$ or $m$ odd.

Since $\ell$ and $m$ are nonzero, 
$nx^2 \equiv y^n \pmod 6$.  This with
$\gcd(nx,y) = 1$ yields
\[n\equiv y \equiv \pm 1 \pmod 6\hspace{.15in}\mbox{and}\hspace{.15in}
x\equiv \pm 1 \pmod 6.\]
Since $n > 1$, this implies that, in fact, $n \geq 5$.

Let $\ell = 2k + e$ and $m = 2k^\prime + e^\prime$  with
$k$, $k^\prime \geq 0$ and $e$, $e^\prime\in\{0,1\}$.
Set $w = 2^e3^{e^\prime} \in \{1,2,3,6\}$.
By assumption, $\ell$ and $m$ cannot both be even.
Hence, $w \in \{2,3,6\}$.

Set $a = 2^k3^{k^\prime}\sqrt{w} + x\sqrt{-n}$ and
$b = 2^k3^{k^\prime}\sqrt{w} - x\sqrt{-n}$.
Then $ab = y^n$.  Letting $E = \Q(\sqrt{w},\sqrt{-n})$ and 
$F = \Q(\sqrt{-wn})$, we have 
$a$, $b\in \ri_E$ and 
$a^2$, $b^2\in \ri_F$. 

Suppose that there exists a prime ideal $\fr{p} \subseteq \ri_E$
such that $\fr{p}\mid a\ri_E$ and $\fr{p}\mid b\ri_E$.  Then, since
$\fr{p}\mid ab\ri_E$, $\fr{p}\mid y\ri_E$ and, since $\fr{p}\mid (a+b)\ri_E$
and $w\ri_E\mid 6\ri_E$,
$\fr{p}\mid 6\ri_E$.  But this is not possible, 
since $y$ is relatively prime to $6$ in $\Z$.  
Hence $a\ri_E$ and $b\ri_E$
are relatively prime in $\ri_E$.  It follows easily that
$a^2\ri_F$ and $b^2\ri_F$ are relatively prime in $\ri_F$.

Now, $(a^2\ri_F)(b^2\ri_F) = y^{2n}\ri_F = (y\ri_F)^{2n}$.  By
the unique factorization of ideals in $\ri_F$, there exists an
ideal $I\subseteq \ri_F$ such that $a^2\ri_F = I^{2n}$.  Let $s$ be
the order of the ideal class of $I$ in the class group of $\ri_F$.  
Then there exists 
$\alpha\in \ri_F$ such that $I^s = \alpha\ri_F$.
Since $a^2\ri_F$ is 
principal, we have $s\mid 2n$, and so $2n = st$ for some $t\in \N$.
Further, $a^2\ri_F = I^{2n} = (I^s)^t = \alpha^t\ri_F$,
and so there exists a unit $\varepsilon\in\ri_F$ 
such that $a^2 = \varepsilon\alpha^t$. 
Since $F = \Q(\sqrt{-wn})$ with $wn$ square-free and $n \geq 5$, 
$\varepsilon = \pm 1$.

Suppose $t$ is even, so $t = 2t_0$ for some $t_0\in \N$.  Then 
\[\left(\frac{a}{\alpha^{t_0}}\right)^2 = \varepsilon = \pm 1.\]
But, as is easily verified,
$E$ does not contain a square root of $-1$.  So $\varepsilon = 1$ and
$a = \pm\alpha^{t_0} \in F$, contradicting the 
definition of $a$. 
Thus $t$ is odd and so $s$ is even.  
Further, since $t\mid 2n$, $\gcd(t,6) = 1$.

Replacing $\alpha$ with 
$-\alpha$, if necessary, we may assume, without loss of
generality, that $\varepsilon = 1$.
Thus $a^2 = \alpha^t$.  

Suppose that $t = 1$.  Then $s = 2n$ and so
$2n\mid h_F$, the class number of $\ri_F$.
In particular, $2n \leq h_F$.  Let $d = disc(\ri_F)$.
Then $d = -wn$ or $d = - 4wn$.
By the class number formula and a basic bound on $L(1,\chi_d)$~\cite{Nark},
we have
\[h_F = \frac{\sqrt{|d|}}{\pi} L(1,\chi_d)
\leq \frac{\sqrt{|d|}}{\pi} \left(2 + \log{|d|}\vphantom{\sqrt{d}}\right) =
\frac{2\sqrt{|d|}}{\pi} \left(1 + \log{\sqrt{|d|}}\right).\]
Thus, since $|d| \leq 4wn \leq 24n$,
\[2n \leq h_F \leq \frac{2\sqrt{|d|}}{\pi} 
\left(1 + \log{\sqrt{|d|}}\right)
\leq \frac{2\sqrt{24n}}{\pi} \left(1 + \log{\sqrt{24n}}\right)\]
and so
\[\frac{\sqrt{24}}{\pi\sqrt{n}} \left(1 + \log{\sqrt{24n}}\right) \geq 1.\]
Since
$\frac{\sqrt{24}}{\pi\sqrt{51}} \left(1 + \log{\sqrt{24\cdot 51}}\right)
< 1$ and
$\frac{\sqrt{24}}{\pi \sqrt{n}} \left(1 + \log{\sqrt{24n}}\right)$ is
a decreasing function of $n$, for $n \geq 1$,
we have a contradiction for $n > 50$.

For $n \leq 50$ or, equivalently, $wn \leq 300$,
we consult a class number table (for example~\cite[Table 4]{BorSha})
to find that $h_F \leq 22$.  Since $2n \leq h_F$, we have
$n\leq 11$ and so $wn \leq 66$.  Again consulting the table,
we have $h_F \leq 8$ and so $n \leq 4$, a contradiction.

Thus, $t \neq 1$.

Since $t$ is odd, there exists 
$t_1\in \N$, such that $t = 2t_1 + 1$.
Define $\gamma = \frac{a}{\alpha^{t_1}} \in E$.  
Note that
\[\gamma^2 = \frac{a^2}{\alpha^{2t_1}} 
= \frac{\alpha^t}{\alpha^{2t_1}} = \alpha,\]
and therefore, $\gamma\in \ri_E$.

Let $A$, $B\in \Q$ such that 
\[\alpha = A + B\sqrt{-wn}\] 
and note that since $\alpha^t = a^2$, $A$, $B\neq 0$.
Let $A_1$, $B_1$, $C_1$, $D_1\in \Q$ such that 
$\gamma = A_1\sqrt{w} + B_1\sqrt{-n} + C_1\sqrt{-wn} + D_1$.
A simple calculation, using $\gamma^2 = \alpha$, yields
that either $A_1 = B_1 = 0$ or $C_1 = D_1 = 0$.  If the former
holds, then $\gamma \in \ri_F$ and 
$I^{s/2} = \gamma\ri_F$, contrary to the definition of $s$.
Thus
\[\gamma = A_1\sqrt{w} + B_1\sqrt{-n}.\]
Expanding $\gamma^2 = \alpha$ and equating real and imaginary parts
yields 
\begin{equation}
\label{g2=al}
A = A_1^2 w - B_1^2 n \hspace{.4in} \mbox{and} \hspace{.4in}  B = 2 A_1 B_1.
\end{equation}
Since $B \neq 0$, we have $A_1$, $B_1 \neq 0$.

Now, unless $w = 3$ and $n \equiv 1 \pmod 4$, $\ri_F = \Z[\sqrt{-wn}]$.
So $A$, $B \in \Z$.  Further, considering the possible integral
bases for $E$, in this case, $A_1 \in \Z$ and $2B_1 \in \Z$.
But, by equation~(\ref{g2=al}), $B_1^2n = A_1^2 w - A\in \Z$ 
and so $B_1 \in \Z$.

If we do have $w = 3$ and $n \equiv 1 \pmod 4$, then
$\ri_F = \Z[\frac{1 + \sqrt{-3n}}{2}]$ and
$\ri_E = \Z[\sqrt{3},\frac{\sqrt{3} + \sqrt{-n}}{2}]$.
So we have $2A$, $2B$, $2A_1$, $2B_1 \in \Z$.
Further, equation~(\ref{g2=al}) implies that
$A$, $B \in \Z$ if and only if $A_1$, $B_1 \in \Z$.  

Expanding $2^t a^2 = (2 \alpha)^t$, equating real and imaginary parts, yields 
\begin{equation}
\label{2treal}
2^{\ell+t} 3^m - 2^t n x^2 = 
(2A)\sum_{j=0}^{t_1}\binom{t}{2j}(2A)^{t-2j-1}(2B)^{2j}(-wn)^{j}
\end{equation}
and
\begin{equation}
\label{2timag}
2^{k+t+1}3^{k^\prime}x =
(2B)\sum_{j=0}^{t_1}\binom{t}{2j+1}(2A)^{t-2j-1}(2B)^{2j}(-wn)^{j}.
\end{equation}
By equation~(\ref{2treal}), $3\nmid 2A$.  

Suppose that $3\nmid 2Bw$.  From the definition of $w$, $3 \nmid w$ implies
that $k^\prime \neq 0$.  So, reducing 
equation~(\ref{2timag}) modulo $3$ yields
\[0 \equiv \sum_{j=0}^{t_1}\binom{t}{2j+1}(\pm 1)^j \pmod 3,\]
which is impossible, by Lemma~\ref{sums}. 
Thus $3 \mid 2Bw$.

Let $\delta = \overline{\gamma} = A_1\sqrt{w} - B_1\sqrt{-n}$.
Then $\gamma\delta = A_1^2w + B_1^2 n \in \Q \cap \ri_E = \Z$.
Since $(\gamma\delta)^{2t} = (ab)^2 = y^{2n}$ and $2\nmid y$,
we have $2\nmid \gamma\delta$.  

Recall that if $A_1$, $B_1 \notin \Z$, then $w = 3$, $n\equiv 1 \pmod 4$,
and $2A_1 \equiv 2B_1 \equiv 1 \pmod 2$.
Thus $4\gamma\delta = (2A_1)^2w + (2B_1)^2n \equiv 3 + n \pmod 8$.
Since $2\nmid \gamma\delta$ implies that $8\nmid (2\gamma)(2\delta)$,
we have $n \not \equiv 5 \pmod 8$.  Thus, if $A_1$, $B_1 \notin \Z$,
$n \equiv 1 \pmod 8$.

Now,
\[\gamma^t = \left(\frac{a}{\alpha^{t_1}}\right)^t
= \frac{a^{2t_1 + 1}}{\left(\alpha^t\right)^{t_1}}
= \frac{a^{2t_1 + 1}}{a^{2t_1}} = a.\]
It follows that $\delta^t = b$.  
Further, 
\[\frac{\gamma^t + \delta^t}{\gamma + \delta} =
\sum_{j=0}^{t-1}(-\gamma)^j\delta^{t-j-1} \in \Z,\]
since it is an algebraic integer fixed by every automorphism of $E$. 
Thus, since $\gamma^t + \delta^t = a + b = 2^{k+1}3^{k^\prime}\sqrt{w}$, 
we find that, in $\Z$,
\[\left(\frac{\gamma + \delta}{\sqrt{w}}\right)\left|
\left(\frac{\gamma^t + \delta^t}{\sqrt{w}}\right).\right.\]
Simplifying yields $2A_1 \mid  2^{k+1}3^{k^\prime}$.

\begin{lemma} 
$(\gamma,\delta)$ is a $2t$-defective Lehmer pair.
\end{lemma}

\begin{proof}
First recall that $\gamma\delta \in \Z$ and, since $2\nmid \gamma\delta$,
$\gamma\delta \neq 0$.  Further, 
$(\gamma+\delta)^2 = 4A_1^2w \in \Z - \{0\}$.
Suppose that $p\in \Z$ is prime 
such that $p\mid \gcd(\gamma\delta,(\gamma+\delta)^2)$.
Then, since $2A_1 \mid  2^{k+1}3^{k^\prime}$, $p = 2$ or $p = 3$.  But
$(\gamma\delta)^{2t} = (ab)^2 = y^{2n}$ and $\gcd(y,6) = 1$.
Hence no such $p$ exists and therefore
$\gcd(\gamma\delta,(\gamma+\delta)^2) = 1$.
Note that
$\frac{\gamma}{\delta} \in F$, in which the
only roots of unity are $\pm 1$.  It follows that
$\frac{\gamma}{\delta}$ is not a root of unity, since
$A_1$, $B_1\neq 0$.
Thus, $(\gamma,\delta)$ is a Lehmer pair.

Now suppose that $p$ is a prime divisor of $L_{2t}(\gamma,\delta)$.
Then, since
\begin{align*}
L_{2t}(\gamma,\delta) &= 
\frac{\gamma^{2t} - \delta^{2t}}{\gamma^2 - \delta^2} 
= \frac{(\gamma^{t} - \delta^{t})(\gamma^{t} + \delta^{t})}
{(\gamma - \delta)(\gamma + \delta)} 
= L_t(\gamma,\delta) \frac{a + b}{\gamma + \delta} \\
&= L_t(\gamma,\delta) \frac{2^{k+1}3^{k^\prime}\sqrt{w}}{2 A_1 \sqrt{w}}
= \frac{2^{k+1}3^{k^\prime}}{2A_1}L_t(\gamma,\delta),
\end{align*}
we have that $p = 2$, $p = 3$, or $p\mid L_t(\gamma,\delta)$.

Also, $(\gamma^2 - \delta^2)^2 = -16A_1^2B_1^2wn = -4B^2wn$.  
Since $3|2Bw$, $3|(\gamma^2 - \delta^2)^2$.  
Further, if $A_1$, $B_1\in \Z$, then $2|(\gamma^2 - \delta^2)^2$.  
If, instead, $A_1$, $B_1\notin \Z$, then $w = 3$ and 
$n \equiv 1 \pmod 8$.  Thus,
\[4 L_3(\gamma,\delta) = 4\frac{\gamma^{3} - \delta^{3}}{\gamma - \delta}
= 9(2A_1)^2 - (2B_1)^2n \equiv 9 - 1 \equiv 0 \pmod 8,\]  
and so $2\mid L_3(\gamma,\delta)$.
Hence, in any case, $p \mid 
(\gamma^2 - \delta^2)^2L_1(\gamma,\delta)\dots L_{2t-1}(\gamma,\delta)$.
Thus $(\gamma,\delta)$ is a $2t$-defective Lehmer pair.
\end{proof}

By Bilu, Hanrot, and Voutier~\cite[Theorem 1.4]{BHV01}, since
there exists a $2t$-defective Lehmer pair, $2t \leq 30$.  
Then, by Lemma~\ref{VoThm1},
the only candidates for $\gamma$ with $2t > 12$ are
of the form $\gamma = i^k (\sqrt{3} \pm \sqrt{-n})/2$  with 
$n \equiv 5 \pmod 8$.  But in each of these cases, $A_1$, $B_1\notin \Z$ which,
as shown above, implies that $n \equiv 1 \pmod 8$.
Thus, $2t \leq 12$.  Finally, since
$t \geq 5$ is odd, $t = 5$.

Expanding $a^2 = \alpha^5$ and equating real and imaginary parts, we find
\begin{equation}
\label{a2real}
2^\ell3^m - nx^2 =
A\sum_{j=0}^{2}\binom{5}{2j}A^{5-2j-1}B^{2j}(-wn)^{j}
\end{equation}
and
\begin{equation}
\label{a2imag}
2^{k+1}3^{k^\prime}x =
B\sum_{j=0}^{2}\binom{5}{2j+1}A^{5-2j-1}B^{2j}(-wn)^{j}.
\end{equation}
Similarly, expanding $a = \gamma^5$ yields
\begin{equation}
\label{areal} 
2^k3^{k^\prime}
= A_1\left(A_1^{4}w^2 - 10A_1^{2}B_1^{2}wn + 5B_1^{4}n^{2} \right)
\end{equation}
and
\begin{equation}
\label{aimag}
x =  B_1\left(5A_1^{4}w^2 - 10A_1^{2}B_1^{2}wn + B_1^{4}n^{2} \right).
\end{equation}

Suppose, first, that $A_1$, $B_1\in \Z$.  Since $B = 2A_1B_1$, $2|B$.
By equation~(\ref{a2real}), $\gcd(A,6) = 1$ and so, by
equation~(\ref{a2imag}), $2^{k+1}\mid B$.
To see that $3^{k^\prime} \mid B$, suppose that $k^\prime > 0$
and $3\nmid B$.  Reducing equation~(\ref{a2imag}), 
\[0 = B(5A^4 - 10A^2B^2wn + B^4w^2n^2) \equiv B(-1 - wn + w^2) \pmod 3.\]  
Thus, $3\nmid w$ and so $w = 2$.
Hence, $0 \equiv Bn \pmod 3$, a contradiction.  Therefore, if $k^\prime > 0$,
$3\mid B$.  Since $3\nmid 5A^4$, equation~(\ref{a2imag}) implies
that $3^{k^\prime} \mid B$.

By equation~(\ref{aimag}), $\gcd(B_1,6) = 1$.  Since $B = 2 A_1B_1$ and
$2^{k+1}3^{k^\prime} \mid B$, we have $2^{k}3^{k^\prime} \mid A_1$.
Hence, by equation~(\ref{areal}), 
\[A_1^{4}w^2 - 10A_1^{2}B_1^{2}wn + 5B_1^{4}n^{2} = \pm 1.\]
If $k > 0$, then $2\mid A_1$, and reducing modulo $8$ yields a
contradiction.  If $k = 0$, then we have $2\mid w$ and $2\nmid A_1$.
Again, reducing modulo $8$ yields a contradiction, since 
$2wn \equiv 4 \pmod 8$.

Now suppose that $A_1$, $B_1\notin \Z$.  Then we have 
$w = 3$, $n\equiv 1 \pmod 8$, and 
$(2A_1)^2 \equiv (2B_1)^2 \equiv 1 \pmod 8$.
Equation~(\ref{aimag}) becomes
\begin{align*}
32x &= (2B_1)\left[5(2A_1)^{4}w^2 - 
10(2A_1)^{2}(2B_1)^{2}wn + (2B_1)^{4}n^{2} \right]\\
&= (2B_1)\left[4((2A_1)^2w)^2 + 
((2A_1)^2w - (2B_1)^2n)^2 - 8(2A_1)^2(2B_1)^2wn\right].
\end{align*}
Since $2B_1$ is odd, this implies that 
\begin{equation}
\label{mod32}
4((2A_1)^2w)^2 + ((2A_1)^2w - (2B_1)^2n)^2 - 8(2A_1)^2(2B_1)^2wn
\equiv 0 \pmod{32}.
\end{equation}
Reducing each term:
since $(2A_1)^4w^2 \equiv 1 \pmod 8$, 
we have that $4((2A_1)^2w)^2 \equiv 4 \pmod{32}$;
since $(2A_1)^2w - (2B_1)^2n \equiv 2 \pmod 8$, 
$((2A_1)^2w - (2B_1)^2n)^2 \equiv 4 \pmod{32}$;
and since
$-(2A_1)^2(2B_1)^2wn \equiv 5 \pmod 8$,
$-8(2A_1)^2(2B_1)^2wn \equiv 8 \pmod{32}$. 
Thus, reducing congruence~(\ref{mod32}), we find
$0 \equiv 4 + 4 + 8 \equiv 16 \pmod{32}$, 
a contradiction, which completes the proof.
\end{proof}

\end{document}